\documentclass[11pt]{amsart}
\usepackage{amsmath}
\usepackage{amssymb}
\usepackage{amscd}
\usepackage{color}
\usepackage{mathtools}

\def\NZQ{\mathbb}               
\def\NN{{\NZQ N}}
\def\QQ{{\NZQ Q}}
\def\ZZ{{\NZQ Z}}
\def\RR{{\NZQ R}}

\newtheorem{Theorem}{Theorem}[section]
\newtheorem{Lemma}[Theorem]{Lemma}
\newtheorem{Corollary}[Theorem]{Corollary}

\newtheorem{Remark}[Theorem]{Remark}

\let\epsilon\varepsilon
\let\phi=\varphi
\let\kappa=\varkappa

\begin{document}

\title{generating sequences of valuations on simple extensions of domains}

\author{Razieh Ahmadian}

\address{Razieh Ahmadian, Department of Mathematics, Faculty of Mathematical Sciences, Shahid Beheshti University, Tehran, Iran.}
\email{ahmadian@sbu.ac.ir}

\author{Steven Dale Cutkosky}

\thanks{The second author was partially supported by NSF grant DMS-2054394}

\address{Steven Dale Cutkosky, Department of Mathematics,
University of Missouri, Columbia, MO 65211, USA}
\email{cutkoskys@missouri.edu}


\subjclass[2000]{12J20, 16W60, 14B25}
\begin{abstract}Suppose that $(K,v_0)$ is a valued field, $f(x)\in K[x]$ is a monic and irreducible polynomial and $(L,v)$ is an extension of valued fields, where $L=K[x]/(f(x))$.  Let $A$ be a local domain with quotient field $K$ dominated by the valuation ring of $v_0$ and such that $f(x)$ is in $A[x]$. The study of these extensions is a classical subject. 

This paper is devoted to the problem of describing the structure of the associated graded ring ${\rm gr}_v A[x]/(f(x))$ of $A[x]/(f(x))$ for the filtration defined by $v$ as an extension of the associated graded ring of $A$ for the filtration defined by $v_0$. 
We give a complete simple description of this algebra when there is unique extension of $v_0$ to $L$ and the residue characteristic of $A$ does not divide the degree of $f$. To do this, we show that the sequence of key polynomials constructed by MacLane's algorithm can be taken to lie inside $A[x]$. This result was proven using a different method in the more restrictive case that the residue fields of $A$ and of the valuation ring of $v$ are equal and algebraically closed in a recent paper by Cutkosky, Mourtada and Teissier. 
\end{abstract}

\maketitle 

\section{Introduction}
 
Suppose that $(K,v_0)$ is a valued field, $f(x)\in K[x]$ is a monic and irreducible polynomial and $(L,v)$ is the finite field extension $L=K[x]/(f(x))$ where $v$ is an extension of $v_0$ to $L$. Further suppose that $A$ is a local domain with quotient field $K$ such that $v_0$  dominates $A$ and that $f(x)\in A[x]$. 

 The valuations $v_0$ and $v$  induce filtrations of $K$ and $L$ respectively and the associated graded ring of $L$ along $v$ is an extension of the associated graded ring of $K$ along $v_0$. These rings  have been constructed implicitly, in the papers \cite{M}, \cite{M1} of MacLane 
for discrete rank one valuations, and for general valuations by Vaqui\'e in \cite{V}, \cite{V2}, \cite{V1}.  Further papers on this topic, and comparison with the method of pseudo convergent sequences (introduced by Ostrowski in \cite[Teil III, \S 11]{O} and developed by Kaplansky in \cite{Ka}) are \cite{RB}, \cite{NS}, \cite{Sa}, \cite{HMOS} and \cite{DMS}. Finding generating sequences for $A[x]/(f(x))$ in the case where $A$ is no longer a field but an arbitrary noetherian subring dominated by the valuation ring $\mathcal O_{v_0}$ of $v_0$ and with the same field of fractions is much more closely related to resolution of singularities via local uniformization and correspondingly more difficult.
In this paper we study this problem. 

We will use the notation that $\mathcal O_{v_0}$ is the valuation ring of $v_0$, with maximal ideal $m_{v_0}$ and residue field $Kv_0=\mathcal O_{v_0}/m_{v_0}$ and $v_0K$ is the valuation group of $v_0$.

 Given a subring $A$ of $K$,  the associated graded ring of $A$ along $v_0$ is defined as 
$$
{\rm gr}_{v_0}(A)=\bigoplus_{\gamma\in v_0K}\mathcal P_{\gamma}(A)/\mathcal P_{\gamma}^+(A)
$$
where 
$$
\mathcal P_{\gamma}(A)=\{g\in A\setminus \{0\}\mid v_0(g)\ge\gamma\}\mbox{ and }
\mathcal P_{\gamma}^+(A)=\{g\in A\setminus \{0\}\mid v_0(g)>\gamma\}
$$
are valuation ideals of $A$.
The initial form ${\rm In}_{v_0}(g)$ of an element $g\in A$ in ${\rm gr}_{v_0}(A)$ is the class of $g$ in 
$\mathcal P_{v_0(g)}(A)/\mathcal P_{v_0(g)}^+(A)$, the graded component of degree $v_0(g)$ of ${\rm gr}_{v_0}(A)$.

The ring ${\rm gr}_{v_0}(A)$ is an algebra over its degree zero subring. It is a domain which is generally not Noetherian. In this text we shall consider subrings $A$ of $\mathcal O_{v_0}$ so that the semigroup $S^A(v_0)$ of values of elements of $A\setminus\{0\}$ which indexes the homogeneous components of ${\rm gr}_{v_0}(A)$ is contained in the positive part of $v_0K$. 

Important invariants of a finite extension $(K,v_0)\subset (L,v)$ of valued fields are the reduced ramification index and residue degree of $v$ over $v_0$, which are
$$
e(v/v_0)=[vL:v_0K]\mbox{ and }f(v/v_0)=[Lv:Kv_0].
$$
Another, very subtle invariant is the defect $\delta(v/v_0)$ of the extension, which is a power of the characteristic $p$ of the residue field $Kv_0$.  The defect and its role in local uniformization is explained in \cite{K1}.  In the case where $v$ is the unique extension of $v_0$ to $L$ we have by Ostrowski's Lemma that
\begin{equation}\label{eqN300}
[L:K]=e(v/v_0)f(v/v_0)\delta(v/v_0).
\end{equation}
If $A$ and $B$ are local  domains with quotient fields $K$ and $L$ such that $v$ dominates $B$ and $B$ dominates $A$, we have a graded inclusion of graded domains
$$
{\rm gr}_{v_0}(A)\rightarrow {\rm gr}_{v}(B).
$$
The index of quotient fields is:
$$
[{\rm QF}({\rm gr}_{v}(B)):{\rm QF}({\rm gr}_{v_0}(A))]=e(v/v_0)f(v/v_0)
$$
by Proposition 3.3  of \cite{C4}.  The defect seems to disappear, but it manifests itself in mysterious behavior in the extensions of associated graded rings of injections $A'\rightarrow B'$ of birational extensions of Noetherian local domains $A, B$. For instance, if $v_0$ has rational rank 1 but is not discrete, the defect $\delta(v/v_0)$ is larger than 1 and $A$ and $B$ are two dimensional excellent local domains, then ${\rm gr}_{v}(B')$ is not a finitely generated ${\rm gr}_{v_0}(A')$-algebra for any regular local rings $A'\rightarrow B'$ which are dominated by $v$ and dominate $A$ and $B$ as shown in \cite{C3}.
The  construction of  generating sequences is closely related to the problem of local uniformization. In \cite[Theorem 7.1]{CM}, it is shown how reduction of multiplicity along a rank 1 valuation can be achieved in a defectless  extension $A\rightarrow A[x]/(f(x))$. 
A similar statement is proven by San Saturnino in \cite{Sa}.


The subring of degree zero elements of the graded ring ${\rm gr}_{v_0}(A)$ is $({\rm gr}_{v_0}(A))_0=A/Q$ where $Q$ is the prime ideal in $A$ of elements of positive value.
A generating sequence for  $v_0$ on  $A$ is an ordered set of elements of $A$ whose classes in ${\rm gr}_{v_0}(A)$ generate ${\rm gr}_{v_0}(A)$ as a graded $({\rm gr}_{v_0}(A))_0$-algebra.  To be meaningful, a generating sequence should come with a formula for computing the values of elements of $A$, and their relations in ${\rm gr}_{v_0}(A)$. In particular, a generating sequence should give the structure of ${\rm gr}_{v_0}(A)$ as a graded $({\rm gr}_{v_0}(A))_0$-algebra. 

In the case of an inclusion $A\subset B$ of domains, and an extension $v$ of $v_0$ to the quotient field of $B$ such that $v$ has nonnegative value on $B$, a generating sequence of the extension is an ordered sequence of elements of $B$ whose classes in ${\rm gr}_{v}(B)$ generate ${\rm gr}_{v}(B)$ as a ${\rm gr}_{v_0}(A)$-algebra. A generating sequence for an extension should come with a formula for computing the values of elements of $B$, relative to the values of elements of $A$, and give their relations in ${\rm gr}_{v}(B)$. That is, a generating sequence should give the structure of ${\rm gr}_{v}(B)$ as a graded ${\rm gr}_{v_0}(A)$-algebra. 

Let $A$ be a local domain, and $v_0$ be a valuation of the quotient field $K$ of $A$ which dominates $A$. Suppose that $f(x)\in A[x]$ is monic and irreducible in $K[x]$. Suppose that the characteristic $p$ of the residue field of $A$ does not divide $\deg f$ and there is a unique extension $v$ of $v_0$ to the quotient field of $A[x]/(f(x))$. 
We prove in Theorem \ref{Theorem2} of this paper that a good generating sequence $\phi_1,\ldots,\phi_{n-1}$ exists for the extension $A\rightarrow A[x]/(f(x))$. We deduce in the corollary to Theorem \ref{Theorem2} that ${\rm gr}_{v}(A[x]/(f(x)))$ is a finitely presented ${\rm gr}_{v_0}(A)$-algebra.

We show that we have  a graded isomorphism
 $$
 {\rm gr}_{v}(A[x]/(f(x)))\cong {\rm gr}_{v_0}(A)[\overline\phi_1,\ldots,\overline \phi_{n-1}]/I
 $$
 where $\overline\phi_i$ is the initial form of $\phi_i$ in ${\rm gr}_v(A[x]/(f(x)))$ and 
 $$
 I=(\overline\phi_1^{m_1}+\sum \overline c_{1,k}\overline \phi_1^k,\ldots,
 \overline \phi_{n-1}^{m_{n-1}}+\sum \overline c_{n-1,j_1,\ldots, j_{n-2},k}\overline \phi_1^{j_1}\cdots \overline \phi_{n-2}^{j_{n-2}}\overline\phi_{n-1}^k),
 $$
  with $\overline c_{i,j_1,\ldots,j_{i-1},k}\in {\rm gr}_{v_0}(A)$.

Theorem \ref{Theorem2} is proven in Theorem 5.1 of \cite{CMT} with the additional assumptions that  $A/m_A=Kv_0$ is algebraically closed. 
The proof in \cite{CMT} is quite different from the one given here since we make use there of a very explicit  construction of  key polynomials as binomials, which is only possible when the residue field is algebraically closed. Some difficulties in extending the proof of \cite{CMT} to the case of a non algebraically closed field are discussed in Section \ref{SectionEx}.
In Theorem 5.1 \cite{CMT},  It is shown  that with the extra assumptions that $A/m_A=Kv_0$ is algebraically closed,  
  $$
  {\rm gr}_v(A[x]/(f(x)))\cong {\rm gr}_{v_0}(A)[\overline \phi_1,\ldots,\overline\phi_{n-1}]/I
  $$
  where $I$ is generated by binomials,  
  $$
  I=(\overline \phi_1^{m_1}+\overline c_1, \overline \phi_2^{m_2}+\overline c_2\phi_1^{j_1(2)},\ldots,
  \overline \phi_{n-1}^{m_{n-1}}+\overline c_{n-1}\overline\phi_1^{j_1(n-1)}\cdots \overline\phi_{n-2}^{j_{n-2}(n-1)}).
  $$

 Since the defect $\delta(v/v_0)$ is always a power of $p$, the assumption that $p$ does not divide the degree of  $f$ in Theorem \ref{Theorem2} and the assumption that $v$ is the unique extension of $v_0$  forces the defect $\delta(v/v_0)$ to be  1 (so the extension in defectless) by (\ref{eqN300}). 
 
We show in \cite{CMT} that if any of the assumptions in Theorem \ref{Theorem2} are removed, then the conclusions of Theorem \ref{Theorem2} do not hold (Examples of Section 4 and Section 11 of \cite{CMT}).  For instance,  the assumption that $A[x]/(f(x))$ is a ``hypersurface singularity'' is shown to be necessary for finite generation to hold in Example 11.3 \cite{CMT}.

 \section{MacLane's method of constructing valuations}\label{SecMacMeth}
 In this section, we give a quick survey of MacLane's theory of key polynomials. This section is similar to the survey in Sections 2 and 3 of \cite{CMT}. This section is included here to establish notation and for the reader's convenience.

 In this paper, a local ring is a commutative ring with a unique maximal ideal. In particular, we do not require a local ring to be Noetherian. We will denote the maximal ideal of a local ring $A$ by $m_A$.  The quotient field of a domain $A$ will be denoted by ${\rm QF}(A)$. We will say that a local ring $B$ dominates a local ring $A$ if $A\subset B$ and $m_B\cap A=m_A$.
 Suppose that $A$ is a Noetherian local domain with quotient field $K$ and $A\rightarrow A_1$ is an extension of local domains such that  $A_1$ is a domain whose quotient field is $K$ and $A_1$ is essentially of finite type over $A$ ($A_1$ is a localization of a finitely generated $A$-algebra). Then we will say that $A\rightarrow A_1$ is a birational extension.
 
 We will denote the natural numbers by $\NN$ and the positive integers by $\ZZ_+$.

Suppose that $v_0$ is a valuation on a field $K$. Its valuation ring $\mathcal O_{v_0}$, residue field $v_0K$ and value group $Kv_0$ are defined in the introduction.

 If $A$ is a domain which is contained in $\mathcal O_{v_0}$, then the associated graded ring of $A$ along $v_0$ is
 ${\rm gr}_{v_0}(A)$ as defined in the introduction. 
  A pseudo valuation $w$ on a domain $A$ is a surjective map $w:A\rightarrow G_w\cup\{\infty\}$ where $G_w$ is a totally ordered Abelian group and a prime ideal
 $$
 I(w)_{\infty}=I^A(w)_{\infty}=\{g\in A\mid w(g)=\infty\}
 $$
 of $A$ such that $w:{\rm QF}(A/I(w)_{\infty})\setminus \{0\}\rightarrow G_w$ is a valuation.

 Suppose that  $v$ is a valuation or a pseudo valuation on a domain $A$. Following MacLane in \cite{M}, 
 we can define an equivalence $\sim_{v}$ on $A$ defined for $g,h\in A$ by $g\sim_{v} h$  if $v(g-h)>\min\{v(g),v(h)\}$ or $v(g)=v(h)=\infty$. If this holds, we say that $g$ is equivalent to $h$ in $v$.
We say that $g\in A$ is equivalence divisible by $h$ in $v$, written $h|_{v}g$, if there exists $a\in A$ such that $g\sim_{v} ah$. An element $g$ is said to be equivalence irreducible in $v$ if $g|_{v}ab$ implies $g|_{v}a$ or $g|_{v}b$. These conditions can be expressed respectively as the statement that ${\rm In}_{v}(h)={\rm In}_{v}(g)$ in ${\rm gr}_{v}(A)$, that ${\rm In}_{v}(h)$ divides ${\rm In}_{v}(g)$ in ${\rm gr}_{v}(A)$ and that the ideal generated by ${\rm In}_{v}(g)$ in ${\rm gr}_{v}(A)$ is prime.

We review MacLane's algorithm \cite{M}  to construct the extensions of a valuation $v_0$ of a field $K$ to a valuation or pseudo-valuation of the polynomial ring $K[x]$. MacLane applied his method to construct extensions of rank 1 discrete valuations of  $K$ to $K[x]$. This algorithm  has been extended to general valuations by Vaqui\'e \cite{V}. An alternative approach to key polynomials is developed in \cite{NS}.

 MacLane constructs ``augmented sequences of inductive valuations''
\begin{equation}\label{eqM1}
v_1,\ldots,v_k,\ldots
\end{equation}
which extend $v_0$ to $K[x]$. An augmented sequence (\ref{eqM1}) is constructed from successive inductive valuations
\begin{equation}\label{eqM2}
v_k=[v_{k-1};v_k(\phi_k)=\mu_k]\mbox{ for $1\le k$}
\end{equation}
of $K[x]$, where $\phi_k$ is a ``key polynomial'' over $v_{k-1}$  and $\mu_k$ is a ``key value'' of $\phi_k$ over $v_{k-1}$. We always take $\phi_1=x$.

We say that $\phi(x)\in K[x]$ is a key polynomial with key value $\mu$ over $v_{k-1}$ if
\begin{enumerate}
\item[1)] $\phi(x)$ is equivalence irreducible in $v_{k-1}$.
\item[2)] $\phi(x)$ is minimal in $v_{k-1}$; that is, if $\phi(x)$ divides $g(x)$ in $v_{k-1}$, then $\deg_x\phi(x)\le\deg_xg(x)$.
\item[3)] $\phi(x)$ is monic and $\deg_x\phi(x)>0$.
\item[4)] $\mu>v_{k-1}(\phi(x))$.
\end{enumerate}
Following MacLane (\cite[Definition 6.1]{M}) we also assume
\begin{enumerate}
\item[5)] $\deg_x\phi_i(x)\ge\deg_x\phi_{i-1}(x)$ for $i\ge 2$.
\item[6)] $\phi_i(x)\sim \phi_{i-1}(x)$ in $v_{i-1}$ is false. Here the equivalence is to be understood for polynomials in $K[x]$.
\end{enumerate}

It follows from \cite[Theorem 9.4]{M} that 
\begin{equation}\label{eqM16}
\mbox{if $\phi(x)$ is a key polynomial over $v_{k-1}$ then $\deg_x\phi_{k-1}(x)$ divides $\deg_x\phi(x)$.}
\end{equation} 

The key polynomials $\phi_k(x)$ can further be assumed to be homogeneous in $v_{k-1}$, which will be defined after (\ref{eqM5}).

MacLane shows that if $v_0$ is discrete of rank 1, then the extensions of $v_0$ to a valuation or pseudo valuation of  $K[x]$ are the $v_k$ arising from  augmented sequences of finite length
 (\ref{eqM1}) and the limit sequences of augmented sequences of infinite length (\ref{eqM1}) which determine a limit value
 $v_{\infty} $ on $K[x]$ defined by 
 $$
 v_{\infty}(g(x))=\lim_{k\rightarrow\infty} v_k(g(x))\mbox{ for }g(x)\in K[x].
 $$
 We have that $v_{\infty}(g(x))$ is well defined whenever  $v_0$ has rank 1, and is a valuation or pseudo-valuation by the argument of  \cite[page 10]{M}.
 
 MacLane's method has been extended by Vaqui\'e \cite{V},  to eventually construct all extensions of an arbitrary valuation $v_0$ of $K$ to a valuation or pseudo valuation of $K[x]$. 
 
 To compute the ``$k$-th stage''  value $v_k(g(x))$  for $g(x)\in K[x]$ by MacLane's method, we consider the unique expansion
 \begin{equation}\label{eqM3}
 g(x)=g_m(x)\phi_k^m+g_{m-1}(x)\phi_k^{m-1}+\cdots+g_0(x)
 \end{equation}
 with $g_i(x)\in K[x]$, $\deg_xg_i(x)<\deg_x\phi_k(x)$ for all $i$ and $g_m(x)\ne 0$. Then
 $$
 v_k(g(x))=\min\{v_{k-1}(g_m(x))+m\mu_k,v_{k-1}(g_{m-1}(x))+(m-1)\mu_k,\ldots,v_{k-1}(g_0)(x)\}.
 $$
 This expression suffices to prove by induction, assuming the existence of a unique expansion of the coefficients $g_i(x)$ in terms of the polynomials $\phi_j(x)$ with $j<k$, that every $g(x)\in K[x]$ has a unique expansion
  \begin{equation}\label{eqM4}
 g(x)=\sum_ja_j\,\phi_1^{m_{1,j}}\phi_2^{m_{2,j}}\cdots \phi_k^{m_{k,j}}
\end{equation}
with $a_j\in K$ and
  $0\le m_{i,j}<\deg_x\phi_{i+1}/\deg_x\phi_i\mbox{ for }i=1,\ldots,k-1$. Recall that $\deg_x\phi_{i+1}/\deg_x\phi_i$ is a positive integer by (\ref{eqM16}).  Then
\begin{equation}\label{eqM5}
v_k(g)=\min_jv_k(a_j\phi_1^{m_{1,j}}\phi_2^{m_{2,j}}\cdots \phi_k^{m_{k,j}}).
\end{equation}
  If all terms in (\ref{eqM4}) have the same values in $v_k$ then $g$ is said to be homogeneous in $v_k$.

  \begin{Remark}\label{RemarkM20}If 
   $A$ is a subring of $K$ such that $\phi_i\in A[x]$ for $1\le i\le k$ and $g\in A[x]$, then the coefficients $a_j$ in (\ref{eqM4}) are all in $A$.
   \end{Remark}
  
  The polynomial $g$, with expansion (\ref{eqM3}), is minimal in $v_k$ if and only if $g_m\in K$ and \begin{equation}\label{eqM13}
  v_k(g)=v_k(g_m\phi_k^m)
  \end{equation}
  by 2.3 \cite{M1} or Theorem 9.3 \cite{M}.

  By 3.13 of \cite{M1} or \cite[Theorem 6.5]{M}, for $k>i$,
  \begin{equation}\label{eqM8}
  v_k(\phi_i)=v_i(\phi_i)\mbox{ and }v_k(g)=v_{i}(g)\mbox{ whenever }\deg_xg<\deg_x\phi_{i+1}.
  \end{equation}
  
 Further, by \cite[Theorems 5.1 and 6.4]{M}, or \cite[3.11 and 3.12]{M1},
 \begin{equation}\label{eqM15}
 \mbox{For all $g\in K[x]$, $v_k(g)\ge v_{k-1}(g)$ with equality if and only if $\phi_k\not\,\mid_{v_k} g$.}
 \end{equation}
  
  Now suppose that $f(x)\in K[x]$ is monic and irreducible. The extensions of $v_0$ to valuations of $K[x]/(f(x))$ are the extensions of $v_0$ to 
pseudo valuations $v$ of $K[x]$ such that $I(v)_{\infty}=(f(x))$. MacLane \cite{M1} gives an explicit explanation of how his algorithm can be applied to construct the pseudo valuations $v$ of $K[x]$ which satisfy $I(v)_{\infty}=(f(x))$ in Section 5 of \cite{M1} (when $v_0$ is discrete of rank 1). Vaqui\'e shows in \cite{V}, \cite{V2} and  \cite{V1} how this algorithm can be extended to arbitrary valuations $v_0$ of $K$.

Suppose that $v_1,\ldots, v_k$ is an augmented  sequence of inductive valuations in $K[x]$. Expand 
$$
f=f_m\phi_k^m+\cdots+f_0
$$
as in (\ref{eqM3}). Define the projection of $v_k$ by 
\begin{equation}\label{eqN1}
{\rm proj}(v_k)=\alpha-\beta
\end{equation}
 where $\alpha$ is the largest and $\beta$ is the smallest amongst the exponents $j$ for which $v_k(f(x))=v_k(f_j\phi_k^j)$. 
A $k$-th approximant $v_k$ to $f(x)$ over $v_0$ is a $k$-th stage homogeneous (meaning that the key polynomial $\phi_i$ is homogeneous in $v_{i-1}$ for $i\le k$)
inductive valuation which is an extension of $v_0$ and which has a positive projection (\cite[Definition 3.3]{M1}).

First approximants $v_1$ to $f$ are defined as $v_1=[v_0;v_1(\phi_1)=\mu_1]$, where $\phi_1=x$ and $\mu_1$ is chosen so that ${\rm proj}(v_1)>0$. MacLane shows in \cite[Lemma 3.4]{M1} that if $v_k$ is a $k$-th approximant to $f(x)$, then so is $v_i$ for $i=1,\ldots,k-1$. Further, $\phi_k|_{v_{k-1}}f$  and $v_k(f(x))>v_{k-1}(f(x))>\cdots>v_1(f(x))$. In \cite[Theorem 10.1]{M1}, MacLane shows that if $v_0$ is a discrete valuation of rank 1 then
every extension of $v_0$ to a valuation of $K[x]/(f(x))$ is an augmented sequence of finite length of approximants
$v_1,\ldots,v_k$ such that $v_k(f(x))=\infty$  or a limit of an augmented sequence of approximants of infinite length such that $v_{\infty}(f(x))=\infty$. If $v_0$ is not discrete of rank 1, then there is the possibility that the algorithm will have to be continued to construct a pseudo valuation $w$ of $K[x]$ with $w(f(x))=\infty$. If this last case occurs, then the situation becomes quite complicated, as we must then extend the family $\{v_k\mid k\in \ZZ_+\}$ to a ``simple admissible family'' and possibly make some jumps. 
This is shown by Vaqui\'e in \cite[Theorem 2.5]{V1}. An essential point is that for every construction
$v_1,\ldots,v_k$ of approximants to $f$ over $v_0$ by MacLane's algorithm, there exists an extension $w$ of $v_0$ to a pseudo valuation of $K[x]$ such that $I(w)_{\infty}=(f(x))$ and $w(\phi_k)=v_k(\phi_k)$ for all $k$. Another development of approximants for arbitrary valuations is given in \cite{NS}.

MacLane gives the following explanation of how to find all of the extensions of a $(k-1)$-st stage  approximant $v_{k-1}$ to $f$ over $v_0$ to a $k$-th stage approximant $v_k$ to $f$ over $v_0$. 

We say that $e\in K[x]$ is an ``equivalence unit'' for $v_k$ if there exists an ``equivalence-reciprocal''  $h\in K[x]$ such that $eh\sim_{v_k} 1$. It is shown in Section 4 of \cite{M1} that $e$ is an equivalence unit if and only if $e$ is equivalent in   $v_k$ to a polynomial $g$ such that $\deg_xg<\deg_x\phi_k$.

 By \cite[Theorem 4.2 ]{M1}, $f$ has an essentially unique (unique up to equivalence in $v_{k-1}$) expression

\begin{equation}\label{eqM6}
f\sim_{v_{k-1}} e\phi_{k-1}^{m_0}\psi_1^{m_1}\cdots\psi_t^{m_t},
\end{equation}
 with $m_0\in \NN$ and $m_1,\ldots,m_t>0$.  Here $e$ is an equivalence unit for $v_{k-1}$  and $\psi_1,\ldots,\psi_t$ are homogeneous key polynomials  over $v_{k-1}$ all not equivalent to  $\phi_{k-1}$ in $v_{k-1}$ and not equivalent in $v_{k-1}$ to  each other. We have that $t>0$ since ${\rm proj}(v_{k-1})>0$. We have that $\phi_{k-1}$ is a homogeneous  key polynomial in $v_{k-1}$ by \cite[Lemma 4.3]{M1}.

If $f$ is a homogeneous key polynomial for $v_{k-1}$, then $v_k=[v_{k-1};v_k(f(x))=\infty]$ is a 
pseudo valuation of $K[x]$ with $I(v_k)_{\infty}=(f(x))$.

If $f$ is not a homogeneous key polynomial for $v_{k-1}$, then none of the $\psi_i$ are equal to $f$, and we may define a $k$-th stage approximant to $f$ over $v_0$ which is an inductive valuation of $v_{k-1}$ by $v_k=[v_{k-1};v_k(\phi_k)=\mu_k]$ where $\phi_k$ is one of the $\psi_i$. In the expansion (\ref{eqM3}) of $f$,
$$
f=f_m\phi_k^m+\cdots+f_0
$$
$\mu_k$ must be chosen so that ${\rm proj}(v_k)>0$. All $k$-th stage approximants $v_k$ to $f$ extending $v_{k-1}$ are found by the above procedure.

Let $T=\RR\times v_0K$. Given $\alpha,\beta\in v_0K$ and $q\in \RR$, we have the line
$$D=\{(s,\gamma)\in T\mid q\gamma+\alpha s+\beta=0\}
$$
in $T$. When $q\ne 0$, we define the slope of $D$ to be $-\frac{\alpha}{q}\in v_0K\otimes_{\ZZ}\RR$.
Associated to $D$ are two half spaces of $T$,
$$
H^D_{\ge}=\{(s,\gamma)\in T\mid q\gamma+\alpha s+\beta\ge 0\}
$$
and
$$
H^D_{\le}=\{(s,\gamma)\in T\mid q\gamma+\alpha s+\beta\le0\}.
$$
 Given a subset $A$ of $T$, the convex closure of $A$ is ${\rm Conv}(A)=\cap H$ where $H$ runs over the half spaces of $T$ which contain $A$.

The Newton polygon is constructed as on page 500 of \cite{M1} and page 2510 of  \cite{V1}. These constructions are equivalent but slightly different. We use the convention of \cite{M1}. 
 The possible values $\mu_k$ can be conveniently found from the Newton polygon $N(v_{k-1},\phi_k)$. This is constructed by taking the convex closure in  $T$ of
 $$
 A=\{(m-i,\delta)\mid \delta\ge v_{k-1}(f_i), 0\le i\le m\},
 $$
 where the union is over $i$ such that $f_i\ne 0$. A segment $F$ of ${\rm Conv}(A)$ is a subset $F$ of ${\rm Conv}(A)$ which is defined by $F={\rm Conv}(A)\cap D$ where $D$ is a line of $T$ such that ${\rm Conv}(A)$ is contained in one of the half spaces $H^D_{\ge}$ or $H^D_{\le}$ defined by $D$ and $F={\rm Conv}(A)\cap D$ contains at least two distinct points.

 The  slopes of the segments of $N(v_{k-1},\phi_k)$ whose slope $\mu$ satisfies $\mu>v_{k-1}(\phi_k)$ are the possible values of $\phi_k$. The polygon composed of those segments of slope $\mu$ with $\mu>v_{k-1}(\phi_k)$ is called the principal part of the Newton polygon $N(v_{k-1},\phi_k)$.

In the proof of Theorem 5.1 of \cite{M1}, it is shown that for $1\le i\le t$, the principal polygon of $N(v_{k-1},\psi_i)$  (from(\ref{eqM6})) is 
\begin{equation}\label{eqM9}
\{(s,u)\in N(v_{k-1},\psi_i)\mid s\ge m-m_i\}.
\end{equation}
Further, $m_0$ is the smallest exponent $i$ such that in the expansion $f=\sum f_i\phi_{k-1}^i$ with $\deg_x f_i<\deg_z\phi_{k-1}$, we have that $v_{k-1}(f_i\phi_{k-1}^i)=v_{k-1}(f(x))$.

\begin{Remark}\label{RemarkM17}
 If the coefficients of $f(x)$ are all in the valuation ring $\mathcal O_{v_0}$ of $v_0$, then the coefficients of all key polynomials $\phi_k$ are also in $\mathcal O_{v_0}$, as is established in \cite[Theorem 7.1]{M1}.
\end{Remark}

The following theorem follows from a criterion of \cite{V1}.

\begin{Theorem}\label{Theorem4} Suppose that $v_k$ is a $k$-th approximant to $f$ over $v_0$. Then there exists a pseudo valuation $w$ of $K[x]$ such that $w| K=v_0$, $I(w)_{\infty}=(f(x))$, $w(g)\ge v_k(g)$ for all $g\in K[x]$ and $w(\phi_i)=v_i(\phi_i)$ for $1\le i\le k$.
\end{Theorem}

\begin{proof} As explained in the construction of $v_k$ above, we have that $\phi_k|_{v_{k-1}}f$, and there exists a key polynomial $\psi$ for $v_k$ with $\psi$ not equivalent to $\phi_k$ in $v_k$ and such that $\psi|_{v_k}f$. The theorem now follows from \cite[Theorem 1]{V1}.
\end{proof}

\section{An example illustrating some difficulties with residue fields which are not algebraically closed}\label{SectionEx}

Most algorithms for constructing generating sequences for an extension of a valuation $v_0$ dominating a ring $A$ to a valuation 
$v$ dominating a simple extension $B$ of $A$  involve an inductive construction, building up a sequence of monic key polynomials  $P_i$ in the polynomial ring $A[z]$, where the inductive step requires the existence of a polynomial $U\in A[z]$ which is a monomial in the the $P_j$ with $j<i$ such that $v(P_i^n)=v(U)$ where $n$ is the index of the group of values of Laurent monomials in the $P_j$ for $j<i$ in the group of values of Laurent monomials in the $P_j$ for $j\le i$.

This technique is used for instance in the proofs of   \cite{S},  \cite{CV} and 
 of Theorem 5.1 \cite{CMT}.

We give here an example showing that this technique fails if the assumption that $A/m_A\cong k$ is an algebraically closed field is removed from the assumptions  of Theorem 5.1 \cite{CMT}. Thus a different method will have to be utilized to construct generating sequences with these assumptions. 

This technique is probably applicable to extend the algorithm of Theorem 4.1 of \cite{CMT} to the case where the residue field is not algebraically closed. In Theorem 4.1 of \cite{CMT}, $A$ is restricted to be the valuation ring of $v_0$.

Let $K=\QQ(s,t)$ be a rational function field in two variables over $\QQ$. Define a rank 1 valuation $v_0$ on $K$ which dominates $A=\QQ[s,t]_{(s,t)}$ such that the residue field of the valuation ring is $Kv_0=\QQ$, $v_0(s)=1, v_0(t)=\frac{3}{2}$, and the residue class $\left[\frac{t^2}{s^3}\right]=1$ in $Kv_0 =\QQ$. Such a valuation $\nu_0$ exists, as can be seen from the algorithm of  \cite{S}, \cite{CV} (or \cite{M1}, \cite{V} or \cite{NS}).

Let $f(x)=x^2+s$, and let $v$ be a valuation of $K[x]/(f(x))$ which extends $v_0$. We can identify $v$ with a pseudo valuation of $K[x]$ which extends $v_0$ and such that $v(f)=\infty$. Let $v_1$ be the first approximant to $v$. We have that $v(x)=\frac{1}{2}$, so $v_1$ is the Gauss valuation of $K[x]$ such that 
$$
v_1(\sum a_ix^i)=\min\left\{\nu(a_i)+i\frac{1}{2}\right\}.
$$
We will show that $f$ is a key polynomial for $v_1$. If $f$ is not a key polynomial for $v_1$, then as explained in Section \ref{SecMacMeth}, 
$$
f(x)\sim_{v_1}(x+\phi_1)(x+\phi_2)
$$
with $\phi_1,\phi_2$ in the valuation ring $\mathcal O_{v_0}$ of $v_0$. We then have 
$1=v_0(s)=v_0(\phi_1)+v_0(\phi_2)$ and $v_0(\phi_1+\phi_2)>\frac{1}{2}$ so that $v_0(\phi_1)=v_0(\phi_2)=\frac{1}{2}$ and $\phi_2\sim_{v_0}-\phi_1$. Thus $\phi_1\sim_{v_0}c\frac{t}{s}$ for some $0\ne c\in \QQ$. We thus have $x^2+s\sim_{v_1}x^2-c^2\left(\frac{t}{s}\right)^2$ so $s\sim_{v_0}-c^2\left(\frac{t}{s}\right)^2$. This implies
$1=-c^2\left[\frac{t^2}{s^3}\right]=-c^2$ in $K\nu=\QQ$ which is impossible since $c\in \QQ$. This contradiction shows that $f(x)$ is a key polynomial for $v_1$. 

The above calculation shows that there is a unique extension of $v_0$ to a valuation $v$ of $K[x]/(f(x))$. 

The index $n=[\ZZ v(x)+v_0 K:v_0 K]=1$ but there does not exist $U\in A$ such that $v(x)=v_0(U)$ since $v(x)=\frac{1}{2}$ and the smallest positive value of an element of $A$ is 1.

Even though the technique fails, in this example there is  a nice generating sequence in $A[x]/(f(x))$, namely a generating sequence in $A$, starting with $s,t$ as can be constructed from the algorithm of  \cite{CV} and with the addition of $x$.

\section{The construction of a generating sequence}

We use the notation of MacLane from \cite{M} and \cite{M1}, summarized in the introduction and Section \ref{SecMacMeth}.
 
\begin{Lemma}\label{Lemma1}
Let  $v_0$ be a valuation of a field $K$ and  $f(x)\in \mathcal O_{v_0}[x]$ be monic and irreducible in the polynomial ring $K[x]$. Suppose that the characteristic $p$ of the residue field of $\mathcal O_{v_0}$ does not divide $\deg f$ and there is a unique extension $v$ of $v_0$ to  $L=K[x]/(f(x))$. 
Then there exists a finite sequence of inductive values $v_1,\ldots,v_n$ with respective associated key polynomials $\phi_1=x, \phi_2,\ldots,\phi_n=f$ in $\mathcal O_{v_0}[x]$ such  that  $v_k=[v_{k-1}, v_k(\phi_k)=v(\phi_k)]$ for $1\le k\le n$, $\phi_n=f$,  $v_n(f)=\infty$  and 
 $v$ is the valuation on $K[x]/(f(x))$ induced by the pseudo valuation $v_n$. Further, 
 for all $k<n$, $\deg \phi_k\mid \deg f$  and 
$f\sim_{v_{k-1}}\phi_k^{n_k}$ where $n_k=\frac{\deg f}{\deg \phi_k}$.
\end{Lemma}

\begin{proof} Since $v_0$ has a unique extension to $L$, by Ostrowski's Lemma (\ref{eqN300}), 
$$
[L:K]=e(v/v_0)f(v/v_0)\delta(v/v_0)
$$
where the defect $\delta(v/v_0)$ is a power of $p$. Since $p\nmid \deg f=[L:K]$, we have that $\delta(v/v_0)=1$ and so $[L:K]=e(v/v_0)f(v/v_0)$. By 
\cite[Proposition 2.12]{V2}, there exists a finite sequence of key polynomials $\phi_1=x, \phi_2,\ldots,\phi_n=f$ of the associated pseudo valuation  to $v$ on $K[x]$  with successive approximants $v_k=[v_{k-1}, v_k(\phi_k)=v(\phi_k)]$ such that
$v_n$ is the pseudo valuation of $K[x]$ associated to $v$.
 That is, MacLane's algorithm  terminates in a finite number of steps with the construction of a generating sequence for $v$. We have that $\phi_k(x)\in \mathcal O_{v_0}[x]$ for all $k$ by Remark \ref{RemarkM17} since $f\in \mathcal O_{v_0}[x]$.

We have that, for all $k$,
\begin{equation}\label{eq7}
\deg \phi_k|\deg f
\end{equation}
by (\ref{eqM16}). By (\ref{eqM6}), for each $k$,
$$
f(x)\sim_{v_{k-1}}e(x)\phi_{k-1}(x)^{m_0}\psi_1(x)^{m_1}\cdots\psi_t(x)^{m_t}
$$
where $e(x)$ is an equivalence unit for $v_{k-1}$ with $\deg e(x)<\deg \phi_{k-1}(x)$ and $\psi_1,\ldots,\psi_t$ are key polynomials for $v_{k-1}$. By \cite[Theorem 3.1]{V1}, with our assumptions,
$$
f(x)\sim_{v_{k-1}}e_k(x)\phi_k^{n_k}
$$
for some equivalence unit $e_k(x)$ for $v_{k-1}$ and $n_k$ a positive integer. The positive integer $\mbox{proj }v_k$ of (\ref{eqN1}) is computed from the principal part of the  Newton Polygon associated to $\phi_k$ as explained in Section \ref{SecMacMeth}.
By \cite[Theorem 5.2]{M1}, for all $k<n$, 
since the extension of $v_0$ to $L$ is unique,
$$
\mbox{proj }v_k\deg \phi_k=\deg f
$$
and $n_k=\mbox{proj }v_k$ by the paragraph before equation (7) of Section 5 of \cite{M1}, so 
\begin{equation}\label{eq6}
\deg f=n_k\deg \phi_k.
\end{equation}

The effective degree is defined on page 497 of \cite{M1}. Let $g(x)\in K[x]$ be a polynomial and expand 
$$
g(x)=g_s\phi_{k-1}^s+g_{s-1}\phi_{k-1}^{s-1}+\cdots+g_0
$$
with $\deg g_i<\deg \phi_{k-1}$ for all $i$. The effective degree $D_{\phi_{k-1}}(g)$ is the largest exponent $i$ for which 
$v_{k-1}(g)=v_{k-1}(g_i\phi_{k-1}^i)$. Basic properties of the effective degree are derived on page 497 \cite{M1}. It is shown there that $v_{k-1}$ equivalent polynomials have the same effective degrees and $D_{\phi_{k-1}}(gh) =D_{\phi_{k-1}}(g)+D_{\phi_{k-1}}(h)$.

Since $f(x)\sim_{v_{k-1}}e_k\phi_k^{n_k}$ we have that $D_{\phi_{k-1}}(f)=D_{\phi_{k-1}}(e_k\phi_k^{n_k})$.
Since $\phi_k$ is a key polynomial for $v_{k-1}$, we have an expression 
$$
\phi_k=\phi_{k-1}^l+a_{l-1}\phi_{k-1}^{l-1}+\cdots+a_0
$$
for some $l$ with $\deg a_i<\deg \phi_{k-1}$ for all $i$ with
$$
v_{k-1}(\phi_k)=v_{k-1}(\phi_{k-1}^l)=v_{k-1}(a_0)
$$
by \cite[Theorem 9.4]{M}. Expand 
$$
f(x)=f_s\phi_{k-1}^s+f_{s-1}\phi_{k-1}^{s-1}+\cdots+f_0
$$
with $\deg f_i<\deg \phi_{k-1}$ for all $i$. We have that $f_s=1$ since $f$ is monic and by (\ref{eq7}), we have that
\begin{equation}\label{eqN8}
D_{\phi_{k-1}}(f)=D_{\phi_{k-1}}(e_k\phi_k^{n_k})=D_{\phi_{k-1}}(\phi_k^{n_k})=ln_k=n_k\frac{\deg\phi_k}{\deg \phi_{k-1}}.
\end{equation}
Thus 
$$
\deg f\ge D_{\phi_{k-1}}(f)\deg \phi_{k-1}=n_k\deg\phi_k=\deg f
$$
by (\ref{eq6}). Thus $s=ln_k$.

Suppose that $a,b\in K[x]$ and $\deg a,\deg b<\deg \phi_{k-1}$. Expand $ab=q\phi_{k-1}+r$ with $\deg r<\deg \phi_{k-1}$. As explained on the bottom of page 369 of \cite{M}, if $q\ne 0$, then 
$$
v_{k-1}(ab)=v_{k-2}(ab)=v_{k-2}(r)\le v_{k-2}(q\phi_{k-1})<v_{k-1}(q\phi_{k-1}).
$$
 Thus we have an
expansion
$$
e_k\phi_k^{n_k}\sim_{v_{k-1}} e_k\phi_{k-1}^s+a_{s-1}'\phi_{k-1}^{s-1}+\cdots+a_0'
$$
where $\deg a_i'<\deg \phi_{k-1}$ for all $i$. We have $s=D_{\phi_{k-1}}(f)$ and $f\sim_{v_{k-1}}e_k\phi_k^{n_k}$. Thus
$$
v_{k-1}(f-e_k\phi_k^{n_k})>v_{k-1}(f)=v_{k-1}(\phi_{k-1}^s),
$$
and so
$v_{k-1}(\phi_{k-1}^s-e_k\phi_{k-1}^s)>v_{k-1}(\phi_{k-1}^s)$ and so $e_k\sim_{v_{k-1}} 1$.
In conclusion, we have shown that 
$f(x)\sim_{v_{k-1}}\phi_k^{n_k}$ where $n_k\deg\phi_k=\deg f$ for all $k$.

\end{proof}

We now prove the main result of this paper. This theorem generalizes Theorem 5.1 of \cite{CMT}, where Theorem \ref{Theorem2}  is proven with the additional assumptions on residue fields that $A/m_A=Kv_0$ is algebraically closed. The proof in \cite{CMT} is quite different from the one given here since in \cite{CMT}, we make use  of a very explicit  construction of  key polynomials as binomials, which is only possible when the residue field is algebraically closed. 

\begin{Theorem}\label{Theorem2}
Let $A$ be a local domain, and $v_0$ be a valuation of the quotient field $K$ of $A$ which dominates $A$. Suppose that $f(x)\in A[x]$ is monic and irreducible in $K[x]$. Suppose that the characteristic $p$ of the residue field of $A$ does not divide $\deg f$ and there is a unique extension $v$ of $v_0$ to the quotient field of $A[x]/(f(x))$. 
Then there exists a finite sequence of inductive values $v_1,\ldots,v_n$ with respective associated key polynomials $\phi_1=x, \phi_2,\ldots,\phi_n=f$ such that  $v_k=[v_{k-1}, v_k(\phi_k)=v(\phi_k)]$ for $1\le k\le n$, $v$ is the valuation on $A[x]/(f(x))$ induced by the pseudo valuation $v_n$ and 
  each $\phi_i\in A[x]$.
\end{Theorem}

\begin{proof} By Lemma \ref{Lemma1}, there exists a finite sequence of inductive values $v_1,\ldots,v_n$ with respective associated key polynomials $g_1=x, g_2,\ldots,g_n=f$ in $\mathcal O_{v_0}[x]$ such that  $v_k=[v_{k-1}, v_k(g_k)=v(g_k)]$ for $1\le k\le n$, 
$v$ is the valuation on $A[x]/(f(x))$ induced by the pseudo valuation $v_n$. Further, 
 for all $k<n$, $\deg g_k\mid \deg f$  and 
$f\sim_{v_{k-1}}g_k^{n_k}$ where $n_k=\frac{\deg f}{\deg g_k}$. The $g_k$ are homogeneous in $v_{k-1}$ for all $k$. We will inductively construct homogeneous key polynomials $\phi_k$ for $v_{k-1}$ for $1\le k\le n$ such that $\phi_k\sim_{v_{k-1}} g_k$ and $\phi_k\in A[x]$ for all $k$. Set $\phi_1=g_1=x$ and suppose by induction that we have constructed such $\phi_i$ up to $\phi_k$. Set $g=g_{k+1}$, $e=n_{k+1}$.
Expand
$$
g=\sum_{i=0}^ra_i\phi_k^i
$$
with $a_i\in \mathcal O_{v_0}[x]$ and $\deg a_i<\deg \phi_k$  for all $i$ (by Remark \ref{RemarkM17}).  Further, $a_r=1$ and
\begin{equation}\label{eqN9}
D_{\phi_k}(g)=r
\end{equation}
by \cite[Theorem 9.4]{M}. For $0\le i\le re$, set 
\begin{equation}\label{eq3}
c_i=\sum \binom{e}{l_0,l_1,\ldots,l_r}a_0^{l_0}a_1^{l_1}\cdots a_r^{l_r}
\end{equation}
where the  sum  is over $l_0,l_1,\ldots,l_r\in \NN$ such that 
\begin{equation}\label{eqN10}
l_0+l_1+\cdots+l_r=e
\end{equation}
 and 
 \begin{equation}\label{eqN11}
 l_1+2l_2+\cdots+rl_r=i. 
 \end{equation}
Then
$$
g^e=\sum_{l_0+l_1+\cdots+l_r=e} \binom{e}{l_0,l_1,\ldots,l_r}a_0^{l_0}(a_1\phi_k)^{l_1}\cdots(a_r\phi_k^r)^{l_r}\\
=\sum_{i=0}^{re}c_i\phi_k^i.
$$
Since $g$ is homogeneous, 
$v_k(a_i\phi_k^i)=v_k(g)$ for all $0\le i\le r$ such that $a_i\ne 0$. Thus 
$$
v_{k-1}(a_0^{l_0}a_1^{l_1}\cdots a_r^{l_r})=v_k(a_0^{l_0}a_1^{l_1}\cdots a_r^{l_r})=v_k(g^e)-iv_k(\phi_k)
$$
for all terms $a_0^{l_0}a_1^{l_1}\cdots a_r^{l_r}$ appearing in the expansion (\ref{eq3}) of $c_i$ which are nonzero. In particular, 
\begin{equation}\label{eq4}
v_{k-1}(c_i)+v_k(\phi_k^i)\ge\min\{v_{k-1}(a_0^{l_0}a_1^{l_1}\cdots a_r^{l_r})+v_k(\phi_k^i)\}
\ge v_k(g^e)
\end{equation}
for all $i$. Thus
$$
v_k(g^e)\ge \min\{v_k(c_i\phi_k^i)\}\ge  \min\{v_{k}(c_i)+v_k(\phi_k^i)\} \ge \min\{v_{k-1}(c_i)+v_k(\phi_k^i)\}\ge v_k(g^e)
$$
 and so
\begin{equation}\label{eq5}
v_k(g^e)=\min\{v_{k}(c_i)+v_k(\phi_k^i)\}
=\min\{v_{k-1}(c_i)+v_k(\phi_k^i)\}.
\end{equation}

Expand $f=\sum_{i=1}^{re} b_i\phi_k^i$ with $b_i\in A[x]$ and $\deg b_i<\deg \phi_k$  for all $i$,
 and $b_{re}=1$ (by Remark \ref{RemarkM20} and Lemma \ref{Lemma1}). Let 
 $F=\sum b_i\phi_k^i$ where the sum is restricted to the $j$ such that $v_{k-1}(b_j)+jv_k(\phi_k)$ is minimal. To simplify notation, we may assume in the following calculations that $f=F$ is homogeneous, since we do not have to concern ourselves with the non relevant higher value terms of $f$.  We have that
 $$
 D_{\phi_k}(f)=D_{\phi_k}(g^e)=eD_{\phi_k}(g)=er
 $$
 by (\ref{eqN9}). Thus
 $$
 v_k(f)=\min\{v_{k-1}(b_i)+v_k(\phi_k^i)\}=rev_k(\phi_k)
 $$
 Write
 $$
 a_0^{l_0}a_1^{l_1}\cdots a_r^{l_r}=\alpha_{l_0,\ldots,l_r}+\beta_{l_0,\ldots,l_r}\phi_k
 $$
 with $\beta_{l_0,\ldots,l_r}\in \mathcal O_{v_0}[x]$ and $\deg \alpha_{l_0,\ldots,l_r}<\deg \phi_k$ for all $l_0,\ldots,l_r$. 
 
 Since $\phi_k$ is a key polynomial over $v_{k-1}$ it is equivalence irreducible for $v_{k-1}$.
 No $a_i$ is equivalence divisible by $\phi_k$ so
 $a_0^{l_0}a_1^{l_1}\cdots a_r^{l_r}$ is not equivalence divisible by $\phi_k$. Thus 
 $$
 v_{k-1}(a_0^{l_0}a_1^{l_1}\cdots a_r^{l_r})=v_{k-1}(\alpha_{l_0,\ldots,l_r})
 $$ 
 by Lemma 4.3 \cite{M} and Lemma 1.1 \cite{V}. Since 
 $$
 v_{k-1}(\beta_{l_0,\ldots,l_r}\phi_k)\ge v_{k-1}(a_0^{l_0}a_1^{l_1}\cdots a_r^{l_r})
 $$
 by Lemma 4.3 \cite{M} and Lemma 1.1 \cite{V}, and $v_k(\phi_k)>v_{k-1}(\phi_k)$ we have that
 $$
 a_0^{l_0}a_1^{l_1}\cdots a_r^{l_r}\sim_{v_k}\alpha_{l_0,\ldots,l_r}.
 $$ 
 
 We thus  have that for $0\le i\le re$, 
 $$
 \begin{array}{lll}
 c_i&=&\sum \binom{e}{l_0,l_1,\ldots,l_r}a_0^{l_0}a_1^{l_1}\cdots a_r^{l_r}=\sum \binom{e}{l_0,l_1,\ldots,l_r}\alpha_{l_0,\ldots,l_r} +\sum   \binom{e}{l_0,l_1,\ldots,l_r}    \beta_{l_0,\ldots,l_r}\phi_k\\
 &\sim_{v_k} &\sum \binom{e}{l_0,l_1,\ldots,l_r}\alpha_{l_0,\ldots,l_r}
 \end{array}
  $$

 Let $\alpha_i=\sum \binom{e}{l_0,l_1,\ldots,l_r}\alpha_{l_0,\ldots,l_r}$ so that $c_i\sim_{v_k}\alpha_i$.

 Now 
 $$
 \sum\alpha_i\phi_k^i\sim_{v_k} \sum c_i\phi_k^i\sim_{v_k}f=\sum b_i\phi_k^i.
 $$
 Thus $v_k(\sum(\alpha_i-b_i)\phi_k^i)>v_k(f)$.
 Since $\deg \alpha_i<\deg \phi_k$ and $\deg b_i<\deg \phi_k$, we have
 that 
 $$
 v_k(\sum(\alpha_i-b_i)\phi_k^i)=\min\{v_{k-1}(\alpha_i-b_i)+iv_k(\phi_k)\}.
 $$
    Thus for all $i$,
 $$
 v_k(\alpha_i-b_i)=v_{k-1}(\alpha_i-b_i)>v_k(f)-iv_k(\phi_k).
 $$
  In summary, since $v_k(f)=re v_k(\phi_k)$ and $c_i\sim_{v_k}\alpha_i$,   we have that 
 \begin{equation}\label{eqN4}
v_k(c_i-b_i)>(re-i)v_k(\phi_k).
\end{equation}
for all $i$.

By the homogeneity assumptions, for all $l$, we have that either $a_l$ is zero or 
\begin{equation}\label{eq1}
v_k(a_l)=(r-l)v_k(\phi_k)
\end{equation}
and 
for all $l$, we have that either $b_l$ is zero or 
\begin{equation}\label{eq2}
v_k(b_l)=(re-l)v_k(\phi_k).
\end{equation}

Suppose that $0\le j<r$. Then for $i=(e-1)r+j$, if $a_0^{l_0}\cdots a_r^{l_r}$ is a nonzero monomial appearing in $c_i$, 
 then equations (\ref{eqN10}) and (\ref{eqN11}) imply 
$$
r-j=re-i=l_0r+l_1(r-1)+\cdots+l_{r-1}
$$
so that
\begin{equation}\label{eqN3}
l_m=0\mbox{ if }m<j
\end{equation}
 and the only monomial for which $l_j\ne 0$ is $a_ja_r^{e-1}=a_j$. For $0\le j\le r-1$,
$$
c_{(e-1)r+j}=ea_{j}+H_j
$$
where 
$a_j=0$ or $v_k(a_j)=(r-j)v_k(\phi_k)$
 and 
$H_j$ is a sum of integers times monomials $M$ in $a_{j+1},\ldots,a_{r}$ of degree $e$ such that 
$$
v_k(M)=(r-j)v_k(\phi_k).
$$
We have that $p\not| e$ since $p\not|\deg f$. 

We necessarily have that $H_{r-1}=0$ by (\ref{eqN3}) since $a_r=1$  and $v_{k-1}(a_{r-1})=v_k(a_{r-1})>0$. Thus by (\ref{eqN4}),
$$
v_k(a_{r-1}-\frac{1}{e}b_{er-1})>v_k(\phi_k).
$$
We have that $\frac{1}{e}\in A$ since $p\not|e$ and $A$ is a local ring. 
By descending induction on $j$, we have that for $0\le j\le r-1$, there exists $u_j\in A[x]$ such that 
$$
v_k(a_j-u_j)>(r-j)v_k(\phi_k)
$$
since $b_i\in A[x]$ for all $i$.
Set $u_r=a_r=1$. For $0\le i\le r$, write $u_i=a_i+z_i$ with $z_i\in \mathcal O_{v_0}[x]$. We have that $v_k(z_i)>(r-i)v_k(\phi_k)$. Thus
 $$
 v_k(\sum_{i=0}^ru_i\phi_k^i-g)=v_k(\sum_{i=0}^rz_i\phi_k^i)
 \ge \min\{v_k(z_i\phi_k^i)\}>rv_k(\phi_k)=v_k(g).
 $$
 Thus $g\sim_{v_k}\sum_{i=0}^ru_i\phi_k^i$ with $\sum_{i=0}^ru_i\phi_k^i\in A[x]$.
 Set $\phi_{k+1}$ to be the homogenization  of $\sum_{i=0}^ru_i\phi_k^i$, which  is a key polynomial for $v_{k}$ which is equivalent to $g$ and is in $A[x]$ by Remark \ref{RemarkM20}.

\end{proof}

 We have constructed the key polynomials $\phi_i$, for $1\le i\le n$ in Theorem \ref{Theorem2} so that $\phi_i\in A[x]$ for $i\ge 1$ and so that they are homogeneous. Thus they have an expansion
 $$
 \phi_{i+1}=\sum_{k=0}^{m_i}h(i)_k\phi_i^k
 $$
 with each $h(i)_k\in A[x]$ satisfying $\deg h(i)_k<\deg \phi_i$ by Remark \ref{RemarkM20}. Further, $h(i)_{m_i}=1$, $v(h(i)_k)=(m_i-k)v(\phi_i)$ for all $i$ such that $h(i)_k\ne 0$ and $h(i)_0\ne 0$ by Theorem 9.4 \cite{M}. 
 
 Since $\phi_{i+1}$ is homogeneous, we have expansions 
 $$
 h(i)_k=\sum c_{i,j_1,\ldots,j_{i-1},k}\phi_1^{j_1}\cdots \phi_{i-1}^{j_{i-1}}
 $$
 where the sum is over $0\le j_i<n_i$ and $c_{i,j_1,\ldots,j_{i-1},k}\in A$ for all $i,j_1,\ldots,j_{i-1},k$ and $c_{j_1,\ldots,j_{i-1},k}=0$ if
 $v(c_{i,j_1,\ldots,j_{i-1},k}\phi_1^{j_1}\cdots \phi_{i-1}^{j_{i-1}})\ne v(\phi_i^{m_i-k})$. 
 
 The ring ${\rm gr}_v(A[x]/(f(x)))$ is a graded ${\rm gr}_{v_0}(A)$-algebra which is generated as a ${\rm gr}_{v_0}(A)$-algebra 
 by the initial forms $\overline\phi_i={\rm In}_v(\phi_i)$ of $\phi_i$ in  the quotient of valuation ideals 
 $$
 \mathcal P_{v(\phi_i)}(A[x]/(f(x)))/ \mathcal P_{v(\phi_i)}^+(A[x]/(f(x)),
 $$
  the graded component of degree $v(\phi_i)$ in ${\rm gr}_{v}(A[x]/(f(x)))$.  As in Theorem 1.7 \cite{V}, We have by induction on $n$, a graded isomorphism
 $$
 {\rm gr}_{v}(A[x]/(f(x)))\cong {\rm gr}_{v_0}(A)[\overline\phi_1,\ldots,\overline \phi_{n-1}]/I
 $$
 where 
 $$
 I=(\overline\phi_1^{m_1}+\sum \overline c_{1,k}\overline \phi_1^k,\ldots,
 \overline \phi_{n-1}^{m_{n-1}}+\sum \overline c_{n-1,j_1,\ldots, j_{n-2},k}\overline \phi_1^{j_1}\cdots \overline \phi_{n-2}^{j_{n-2}}\overline\phi_{n-1}^k),
 $$
  with $\overline c_{i,j_1,\ldots,j_{i-1},k}$ the initial form of  $c_{i,j_1,\ldots,j_{i-1},k}$ in ${\rm gr}_{v_0}(A)$.
  
  In Theorem 5.1 \cite{CMT}, it is further assumed that $A/m_A=Kv_0$ is algebraically closed. It is shown there, that with these extra assumptions, 
  $$
  {\rm gr}_v(A[x]/(f(x)))\cong {\rm gr}_{v_0}(A)[\overline \phi_1,\ldots,\overline\phi_{n-1}]/I
  $$
  where $I$ is generated by binomials,  
  $$
  I=(\overline \phi_1^{m_1}+\overline c_1, \overline \phi_2^{m_2}+\overline c_2\phi_1^{j_1(2)},\ldots,
  \overline \phi_{n-1}^{m_{n-1}}+\overline c_{n-1}\overline\phi_1^{j_1(n-1)}\cdots \overline\phi_{n-2}^{j_{n-2}(n-1)}).
  $$
  
  Suppose that $S\subset T$ are semigroups. $T$ is a finitely generated $S$-module if there exist finitely many elements $z_1,\ldots,z_l\in T$ such that $T=\cup_{i=1}^l(S+z_i)$.
 
 \begin{Corollary}\label{CortoThm2}  Let assumptions be as in Theorem \ref{Theorem2}. Then
 ${\rm gr}_{v}(A[x]/(f(x)))$ is a finitely presented ${\rm gr}_{v_0}(A)$-algebra.  
 
  Let 
 $$
 S^A(v_0)=\{y\in A\mid v_0(y)\ge 0\}\mbox{ and }S^{A[x]/(f(x))}(v)=\{z\in A[x]/(f(x))\mid v(z)\ge 0\}.
 $$
 Then the semigroup $S^{A[x]/(f(x))}(v)$ is a finitely generated $S^A(v_0)$-module.
 \end{Corollary}
 
 \begin{proof} By Theorem \ref{Theorem2} and the above calculations, ${\rm gr}_{v}(A[x]/(f(x)))$ is a finitely presented ${\rm gr}_{v_0}(A)$-algebra and  $S^{A[x]/(f(x))}(v)$ is generated by $v(\phi_1)=v(x),\ldots,v(\phi_{n-1})$ as an $S^A(v_0)$-module. 
 \end{proof}

\end{document}